\documentclass{amsart}\usepackage{amssymb}
\usepackage{longtable,graphics,multirow,ulem}
\usepackage{tikz}
\usepackage{pgf,tikz}
\usetikzlibrary{arrows}
\usepackage[all]{xy}

\newcommand{\bba}{\mathbb{A}}\newcommand{\bbq}{\mathbb{Q}}\newcommand{\bbp}{\mathbb{P}}\newcommand{\bbr}{\mathbb{R}}\newcommand{\bbz}{\mathbb{Z}}
\newcommand{\caa}{\mathcal{A}}\newcommand{\cac}{\mathcal{C}}\newcommand{\cad}{\mathcal{D}}\newcommand{\cah}{\mathcal{H}}\newcommand{\caj}{\mathcal{J}}\newcommand{\call}{\mathcal{L}}\newcommand{\can}{\mathcal{N}}\newcommand{\cau}{\mathcal{U}}\newcommand{\cav}{\mathcal{V}}\newcommand{\cax}{\mathcal{X}}
\newcommand{\frS}{\mathfrak{S}}
\DeclareMathOperator{\cl}{cl}  \DeclareMathOperator{\gal}{Gal}

\DeclareMathOperator{\Ker}{Ker}\DeclareMathOperator{\ns}{NS}\DeclareMathOperator{\pic}{Pic}\DeclareMathOperator{\rk}{rank}\DeclareMathOperator{\spec}{Spec}\newfont{\cyrm}{wncysc10}

\DeclareMathOperator{\jac}{Jac}
\DeclareMathOperator{\divv}{div}
\DeclareMathOperator{\sym}{Sym}
\newcounter{nootje}
\newtheorem{theorem}{Theorem}\numberwithin{theorem}{section}\theoremstyle{plain}\newtheorem{corollary}[theorem]{Corollary}\newtheorem{lemma}[theorem]{Lemma}\newtheorem{proposition}[theorem]{Proposition}\theoremstyle{definition}\newtheorem{definition}[theorem]{Definition}\newtheorem{remark}[theorem]{Remark}\numberwithin{equation}{section}
\begin{document}
\title[Rank of abelian varieties]{Lower bounds for the rank of families of abelian varieties under base change}

\author[M. Hindry]{Marc Hindry}
\address{Institut de Math\'ematiques Jussieu -- Paris Rive Gauche (IMG-PRG)}
\curraddr{UFR de Math\'ematiques, B\^atiment Sophie Germain, Universit\'e Paris Diderot,
Paris 75013, France}
\email{marc.hindry@imj-prg.fr}

\author[C. Salgado]{Cec\'{\i}lia Salgado}
\address{Universidade  Federal do Rio de Janeiro (UFRJ)}
\curraddr{Instituto de Matem\'atica, Cidade Universit\'aria, Ilha do Fund\~ao, Rio de Janeiro, RJ, Brasil}
\email{salgado@im.ufrj.br}
\date{\today}

\subjclass{Primary 11G05, 11G10,11G30, 14D10;  14H40, 14K15}

\keywords{Diophantine geometry, arithmetic and algebraic geometry,  Abelian and Jacobian varieties,  Mordell-Weil group, rank}

\begin{abstract} 
We consider the following question : given a family of abelian varieties $\caa$ over a curve $B$ defined over a number field $k$, how does the rank of the Mordell-Weil group of the fibres $\caa_t(k)$ vary? A specialisation theorem of Silverman guarantees that, for almost all $t$ in $C(k)$, the rank of the fibre is at least the generic rank, that is the rank of $\caa(k(B))$. When the base  curve $B$ is rational, we show, at least in many cases and under some geometric conditions, that there are infinitely many fibres for which the rank is larger than the generic rank. This paper is a sequel to a paper of the second author \cite{sal} where the case of elliptic surfaces is treated.\end{abstract}

\maketitle
\section{Introduction}\label{intro}

Let $\pi :\caa\rightarrow B$ be a {\it family of abelian varieties} defined over a number field $k$, by which we mean a  flat morphism between projective varieties defined over $k$, such that the generic fibre is an abelian variety $A/K$, where $K=k(B)$. We will most of the time assume that $B$ is a curve and, in fact, the most interesting case will be when $B\cong\bbp^1_k$. Let $t$ be a geometric point of $B$, if $t$ belongs to a dense Zariski open subset $U\subset B$, the fibre $\caa_t:=\pi^{-1}\{t\}$ is again an abelian variety defined over $k(t)$. The Mordell-Weil theorem states that for any $t$, the group $\caa_t(k(t))$ is finitely generated; the Lang-N\'eron theorem asserts that $\caa(k(B))=A(K)$ is also finitely generated. We want to compare the rank of these groups.

The first tool for this is the use of the natural specialisation map $i_t:A(K)\rightarrow \caa_t(k(t))$, which we describe, when $B$ is a smooth curve. Since $\caa$ is projective and $B$ smooth, any point $P\in A(K)$ extends to a section $B\rightarrow \caa$ which can be restricted  to the fibre $\caa_t$. In symbol $i_t$ is the composite:
$$A(K)\cong \caa(B)\rightarrow \caa_t(k(t)).$$
The following results are due respectively to N\'eron and Silverman.

\begin{theorem}\label{specialise} Let $\pi :\caa\rightarrow B$ be a family of abelian varieties defined over a number field $k$; for any good point $t\in B(k)$ denote $i_t:A(K)\rightarrow \caa_t(k(t))$ the specialisation map.
\begin{enumerate}
\item (N\'eron \cite{ne}) The set of points where $i_t$ is not injective is a thin set.
\item (Silverman \cite{sil83}) Assume $B$ is a curve, the set of points where $i_t$ is not injective is a  set of bounded height. In particular the set of rational
points $t\in B(k)$ where the specialisation map fails to be injective is finite.
\end{enumerate}
\end{theorem}

As suggested by the name, the complement of a  {\it thin} set in $\bbp^n(k)$ or $\bba^n(k)$ (see \cite{semw} for a precise definition) contains most of the points of $\bbp^n(k)$ or $\bba^n(k)$. Silverman's theorem is stated with the hypothesis that $\caa$ has no constant part (i.e. the $K/k$-trace is zero), but one may readily dispense with it (see section \ref{heightsection}).
The following corollary is immediate from Silverman's theorem.

\begin{corollary}  Let $\pi :\caa\rightarrow B$ be a family of abelian varieties defined over a number field $k$, with $B$ a curve. For all but finitely many points $t\in B(k)$ we have
\begin{equation}\label{rklb}
\rk \caa_t(k)\geq \rk A(K).
\end{equation}
\end{corollary}

Notice that the conclusion of the corollary  is only interesting when $B(k)$ is infinite, i.e. when $B\cong\bbp^1_k$ or is an elliptic curve with positive rank.
Granting that $B(k)$ is infinite, a natural question is whether the inequality (\ref{rklb}) is almost always an equality or if a strict inequality can be proven for infinitely many $t$'s in $B(k)$. This is the main question addressed by one of the authors in \cite{sal}, for families of elliptic curves, i.e. when $\caa$ is an  elliptic surface or the relative dimension is $g=1$. We study in this paper the case when the relative dimension is arbitrary.
Our first contribution is to show the following.

\begin{definition} A family of abelian varieties $\pi :\caa\rightarrow B$  defined over a number field $k$ satisfies condition ($\cah_1$) when the variety $\caa$ is $k$-unirational. The condition implies $B\cong\bbp^1_k$.
\end{definition}
\begin{theorem}\label{theo1}
Let $\pi :\caa\rightarrow B$ be a family of abelian varieties defined over a number field $k$, with $B\cong\bbp^1_k$. Assume that the family satisfies condition ($\cah_1$). For infinitely many points $t\in B(k)$ we have
\begin{equation}\label{rklb1}
\rk \caa_t(k)\geq \rk A(K)+1.
\end{equation}
\end{theorem}

This will be achieved by showing a more precise statement. When $\psi :C\rightarrow B$ is a finite morphism of smooth projective curves, we denote
$\caa_C\rightarrow C$ a desingularisation of the fibered product $\caa\times_BC\rightarrow C$.

\begin{theorem}\label{theo1bis}
Let $\pi :\caa\rightarrow B$ be a family of abelian varieties defined over a number field $k$, with $B\cong\bbp^1_k$. Assume that the family satisfies condition ($\cah_1$). There exists a curve $C/k$ with a finite covering $\psi: C\rightarrow B$ such that
\begin{enumerate}
\item We have the inequality $\rk \caa_C(k({C}))\geq \rk A(K)+1$.
\item The set of rational points $C(k)$ is infinite.
\end{enumerate}
\end{theorem}

\begin{remark} A natural question is : given a family of abelian varieties $\pi:\caa\rightarrow B\cong\bbp^1_k$, under which conditions does the conclusion of Theorem \ref{theo1} ``for infinitely many  $t\in B(k)$, the rank of $\caa_t(k)$ is greater than the rank of $\caa(k(B))$" still holds? An obvious necessary condition is that $\caa$ should not be constant, i.e. $k$-birational to $A_0\times_kB$, for some abelian variety $A_0/k$. This seems to be a sufficient condition, but we are unable to prove this. Notice that the condition $(\mathcal{H}_1)$ excludes this trivial case and in fact implies more, namely that $\caa$ has no constant part, i.e. its $K/k$ trace is trivial. However having a trivial $K/k$ trace is not a necessary condition.
\end{remark}

\begin{remark} One can raise questions about the density in $\bbp^1(k)$ of points above which the fibre has rank greater than the generic rank, it is in general conjectured that there is a positive proportion  of such points. Ordering points in $\bbp^1(k)$ by their multiplicative height, we recall (see for example \cite{hs}, Theorem B.6.2) that
$\#\{x\in \bbp^1(k)|\;H_k(x)\leq T\}\sim c_kT^2$. It is therefore expected that, in the situation of Theorem  \ref{theo1bis} we should have
$$\#\{x\in \bbp^1(k)|\;H_k(x)\leq T\;{\rm and}\; \rk \caa_C(k({C}))\geq \rk A(K)+1\}\sim c_1T^2?$$
Our arguments yields the following weaker statement:
$$\#\{x\in \bbp^1(k)|\;H_k(x)\leq T\;{\rm and}\; \rk \caa_C(k({C}))\geq \rk A(K)+1\}\geq c_2T^{2/d},$$
where $d$ is the degree of the map $\psi$ in Theorem \ref{theo1bis}.
\end{remark}

Clearly condition ($\cah_1$) can be weakened asking for the same property for an abelian subfamily; similarly we will see that the proof of Theorem \ref{theo1} or \ref{theo1bis} only requires the existence of a unirational surface inside $\caa$ projecting onto the base (or even only the existence of a pencil of rational curves). Instead of formally stating this remark, we will develop it in the context of the
very interesting examples of families of abelian varieties provided by {\it families of Jacobian varieties}. 

\medskip

Starting from a family of curves of genus $g\geq 1$, that is a flat morphism from a surface to a curve $\pi :\cax\rightarrow B$ with generic fibre $X/K$ a curve of genus $g$, we obtain a family of abelian varieties $\caj\rightarrow B$ whose fibre  at $t\in B$, for $t$ lying outside a finite set, is the jacobian variety of the fibre $\cax_t$. We will denote $\caj$ or $\caj_{X}$ this family.

\begin{definition} A family of curves of genus $g\geq 2$, denoted  $\pi :\cax\rightarrow B$  defined over a number field $k$ satisfies condition ($\cah_2$) when the surface $\cax$ possesses a $B$-morphism generically finite of degree $d\leq g$ onto a $k$-unirational surface $\cax_0$ and moreover such that the generic fibre of the morphism $\cax_0\rightarrow B$ is a curve of genus greater or equal to 1. In other words, there is a $k$-morphism $f$ onto a unirational surface $\cax_0$ :
$$\begin{matrix}\cax&\buildrel{f}\over{ \rightarrow}& \cax_0&\hookleftarrow&(\cax_0)_{\eta}\cr
\downarrow&&\downarrow&&\downarrow\cr
B&=&B&\hookleftarrow&\spec(k(B))\cr
\end{matrix}
$$
which has generic degree $d\leq g$ and $(\cax_0)_{\eta}$ has genus $\geq 1$.
The unirationality of $\cax_0$ implies that $B\cong\bbp^1_k$.
\end{definition}
\begin{theorem}\label{theo2}
Let $\pi :\cax\rightarrow B$ be a family of curves of genus $g\geq 1$ defined over a number field $k$, with $B\cong\bbp^1_k$. Assume that the family satisfies condition ($\cah_2$). 
For infinitely many points $t\in B(k)$ we have
\begin{equation}\label{rklbjac1}
\rk \caj_t(k)\geq \rk J(K)+1.
\end{equation}
\end{theorem}

This will be achieved by showing more precisely.

\begin{theorem}\label{theo2bis}
Let $\pi :\cax\rightarrow B$ be a family of curves of genus $g\geq 1$ defined over a number field $k$, with $B\cong\bbp^1_k$. Assume that the family satisfies condition ($\cah_2$). There exists a curve $C/k$ with a finite covering $C\rightarrow B$ such that
\begin{enumerate}
\item Let $\cax_C:=\cax\times_{B}C$ and $\pi_C:\cax_C\rightarrow C$ the projection. we have the inequality $\rk \caj_{\cax_C}(k({C}))\geq \rk J(K)+1$.
\item The set of rational points $C(k)$ is infinite.
\end{enumerate}
\end{theorem}

Throughout this paper, until stated otherwise,  we will assume that the $K/k$-trace of $A/K$ is trivial. This is done to simplify the exposition and is not generally necessary; we explain in a short paragraph (section \ref{kakatrace}) which statements continue to hold in the presence of a non-trivial  $K/k$-trace and how to modify proofs. 
\medskip

\noindent{\bf Examples.} Finally, we provide examples of applications of our techniques and results. 
The first one is that for $g\leq 5$ (the statement for $g\leq 3$ is easier to prove) a ``dense" set of abelian varieties of dimension $g$ satisfies $\rk A(k)\geq 1$ (see Proposition \ref{example1} for a precise statement).
Next we consider the hyperelliptic curve $X_0$ given over $K_0=k(u)$ by the equation
$$y^2=(x-e_1)\dots(x-e_{2g+1})+u$$
and its jacobian $A=\jac(X_0)$; we show that $\rk A(K_0)$ is zero but, if $t^2=u$ and $K=k(t)$, then $\rk A(K)=2g$ and there are infinitely many
$u\in k$ such that $\rk A_u(k)\geq 2g+1$ (see Proposition \ref{example2}).

The final example is the curve $X$ over $K=k(t)$ given in affine equations
$$y^2=(x-e_1)\dots(x-e_{2g_1+1})+t^2\quad{\rm and}\quad z^2=q(x).$$
We show that the jacobian $A=\jac(X)$ has a constant part $A_0$ (also called $K/k$-trace) isogenous to the jacobian of $z^2=q(x)$ and that $\rk J(K)/A_0(k)=4g_1+1$ and again there are infinitely many
$t\in k$ such that $\rk A_t(k)\geq \rk A(K)+1$ (see Proposition \ref{example3}). Under stricter conditions, namely $\deg(p)=3$ and $\deg(q)=4$ and for some root of $p(x)$, say $e_1$, we have $q(e_1)=a_1^2$,  there are even  infinitely many
$t\in k$ such that $\rk A_t(k)\geq \rk A(K)+2$ (see Proposition \ref{example4}).

\smallskip

\noindent{\bf Plan.}  The paper is organised as follows. The next two sections provide preliminaries describing the results from Kummer theory (section \ref{kummer}) and height theory \ref{heightsection} required to prove the main technical result (section \ref{key}, Lemma \ref{prpimportante}), which is then used to prove Theorem \ref{theo1bis}. The fifth section review the theory of families of Jacobians and contains the proof of theorem \ref{theo2bis}. The final section contains a detailed presentation of the examples outlined above and in particular the construction in a special case of a double base change raising the rank by at least two.

\smallskip

\noindent{\bf Aknowledgments.}  The first author thanks the \textit{R\'eseau Franco-Br\'esilien en math\'e\-matiques} for providing support for visiting  Impa and UFRJ in Rio de Janeiro. The second author was partially supported by CNPq grants 305743/2014-7 and 446873/2014-4, and FAPERJ E-26/203.205/2016. Both authors thank Am\'ilcar Pacheco for useful conversations.

\section{Kummer Theory}\label{kummer}

We keep the previous notation, in particular $K=k(B)$ is finitely generated over $\bbq$.
Let $A_m:=\Ker\left\{[m]:A(\bar{K})\rightarrow A(\bar{K})\right\}$ denote the kernel of multiplication by $m$. It is well known that, as groups,
$A_m\cong( \bbz/m\bbz)^{2g}$.

We will use the following lemma which can be proven by transcendence methods (Masser \cite{mas}) or Galois representation theory (Serre \cite{ser1}); the papers quoted provide a proof over number fields, from which an extension to finitely generated fields can readily be asserted.

\begin{lemma}\label{serre} There exists two constants $c=c(A,K)$ and $\gamma=\gamma(g)$, such that for any torsion point $T\in A(\bar{K})$ of order $m$ we have
$$[K(T):K]\geq cm^{\gamma}.$$

\end{lemma}

\begin{definition} Let $m\geq 2$ be an integer and let $A$ be an abelian variety defined over $K$ and $P\in A(K)$ a rational point of infinite order. The point $P$ is {\it indivisible} by $m$ in $A(K)$ if, for any $n$ dividing $m$ and any point $Q\in A(K)$, the equality $nQ=P$ implies $n=\pm 1$.
\end{definition}

The following statement is proven by one of the authors in \cite{hiinv} (Lemme 14), relying on previous work of Ribet and on results of Faltings.
\begin{proposition}\label{kt} Let $A$ be an abelian variety defined over $K$ a finitely generated field of characteristic zero. There is a constant $c_A>0$ with the following property. Let $m$ and $n$ be two integers such that $m$ divides $n$. Let $P\in A(K)$ be a point of infinite order, indivisible by   $m$ and let $Q=\frac{1}{m}P$ be a point in $A(\bar{K})$ such that $mQ=P$ then
$$\left[K\left(\frac{1}{m}P\right):K\right]\geq \left[K\left(\frac{1}{m}P,A_n\right):K(A_n)\right]\geq c_Am.$$
\end{proposition}

\section{Height Theory}\label{heightsection}

Let $A$ be an abelian variety defined over $K=k(B)$ as before and let $L$ be a symmetric ample line bundle on $A$. To the polarization defined by   $L$ we can associate a quadratic form, the N\'eron-Tate height $\hat{h}=\hat{h}_L:A(\bar{K})\rightarrow\bbr$ (see \cite{hs,lang}).

\subsection{The $K/k$-trace}\label{kakatrace}

Let us denote $A_0$ the $K/k$-trace of $A/K$ (for background on the Chow theory, see \cite{brco,lang}), this is an abelian variety $A_0$ defined over $k$ and equiped with a map $\tau:(A_0)_K\rightarrow A$, such that any morphism $\phi:B_K\rightarrow A$, where $B$ is defined over $k$, factors through $\tau$, i.e there is a $k$-morphism $\phi_0:B\rightarrow A_0$ such that $\phi=\tau\circ \phi_{0,K}$. Since we are working in characteristic zero, we may
safely identify $A_0$, or rather $(A_{0})_K$, with an abelian subvariety of $A$, namely $\tau(A_0)_K$.
In fact one may formulate variants of our results replacing the Mordell-Weil group $A(K)$ (resp. $\caa_t(k)$ with the {\it Lang-N\'eron group}  $A(K)/\tau A_0(k)$ (resp. $\caa_t(k)/A_0(k)$). The importance of the $K/k$-trace comes in part from the fact the N\'eron-Tate height is zero on $\tau(A_0)(k)$ so we get a positive definite quadratic form only on  $(A(K)/\tau A_0(k))\otimes\bbr$ (see \cite{brco,lang}). Nevertheless, by using the fact that $A$ is $K$-isogenous to $A_0\times A_1$, where $A_1/K$ is an abelian variety whose $K/k$-trace is trivial, we can most of the time reduce to the case where $A_0=0$.  

\medskip

For example, we restate here Silverman's specialisation theorem (Theorem \ref{specialise}, $(2)$). \begin{proposition} (Silverman \cite{sil83}) Let $k$ be a number field, $K=k(B)$ the function field of a curve defined over $k$ and $A$ an abelian variety defined over $K$. Let $U$ be the Zariski open subset of $B$ over which $A$ has good reduction, then except for finitely many  $t$ in $U(k)$, the specialisation map
$sp_t:A(k({C}))\rightarrow A_t(k)$ is injective.
\end{proposition} 

\begin{proof} This follows readily from \cite{sil83}, except that Silverman assumes that the $K/k$-trace of $A$ is zero. To show that this hypothesis is not needed, observe first that the specialisation map (in the case of equicharacteristic zero) is always injective on torsion, then to deal with the infinite part of $A(k(B))$ we may replace $A$ by an  isogenous abelian variety and therefore assume that $A=A_0\times A_1$ where $A_0$ is the $K/k$-trace and $A_1$ has $K/k$-trace zero. The specialisation map is trivially injective on $A_0(K)\cong A_0(k)$ and is also injective on $A_1(k({C}))$ by Silverman's result.
\end{proof}

\subsection{Canonical heights on families of abelian varieties}

Assuming that the $K/k$-trace, is zero, the height $\hat{h}=\hat{h}_L$ is positive definite on $A(K)\otimes\bbr$. Following Moret-Bailly \cite{mb} we will use the following description of $\hat{h}$. Consider the N\'eron model $\can\rightarrow B$ of $A/K$ and $\caa$ a smooth compactification of $\can$. We pick a divisor $D$ such that $L=O_A(D)$. The line bundle $L$ extends to a line bundle $\call$ on $\caa$ which, after either tensoring with $\bbq$ (i.e. allowing $\call$ to live in to $\pic(\caa)\otimes\bbq$) or replacing $L$ by an adequate multiple $L^{\otimes N}$, we may assume to be cubical ({\it cf. loc. cit.}). If we denote $\cad_0$ the extension of $D$ to $\caa$ by Zariski closure, the line bundle $\call$ may be associated to a divisor $\cad=\cad_0+F$, where $F$ is a sum of components of special fibres of $\can$. For a point $P\in A(K)$, we denote $C_{P}$ the associated section of  $\can\rightarrow B$. The N\'eron-Tate height can then be computed via intersection in $\can$ (see \cite{mb}):
\begin{equation}\label{nth}
\hat{h}({P})=\cad\cdot C_P=\cad_0\cdot C_P+F\cdot C_P
\end{equation}
We note that if we denote $\cad_0'$, $F'$ the extension of $\cad_0$, $F$ to $\caa$ by Zariski closure, we can also write a variant of formula (\ref{nth}) with intersections on $\caa$ as $\hat{h}({P})=\cad'\cdot C_P=\cad_0'\cdot C_P+F'\cdot C_P$. We record what we will need in the following lemma.

\begin{lemma}\label{ntint} The N\'eron-Tate height of a point $P\in A(K)$ depends only on the intersection numbers of $C_P$ with $\cad$ and the components of special fibres of
$\can\rightarrow B$ (resp. with $\cad'$ and the components of special fibres of
$\caa\rightarrow B$).
\end{lemma}

\begin{remark} When $A/K$ is the Jacobian  of a curve $X/K$, it is possible to give a more explicit version of Lemma \ref{ntint} in terms of the configuration of non-reducible fibres of the fibration $\mathcal{X}\rightarrow B$; this is worked out in \cite{shio2,shio3}.

\end{remark}

\subsection{The canonical height and rationally equivalent curves}

Let $\caa$ be as before and $C$ a curve in $\caa$. Suppose that there is a natural number $n$ such that $[n]C=C_0$, where $C_0$ is a section of $\pi: \caa\rightarrow B$. We show that the height of $C_0=[n](C)$ is constant  when $C$ varies in a pencil of curves.

\begin{lemma}\label{numclass}
Let $C$, $C'$ be two members of a pencil $\{C_t\}_{t\in \mathbb{P}^1}$  of curves  in $\caa$. Assume such that there is an integer $n\geq 1$ such that  $[n]C$ and $[n]C'$ are  sections of $\pi$,  and let $P_0$ and $P_0'$ be the corresponding points in $\caa(K)$, then $h(P_0)=h(P_0')$.
\end{lemma}

\begin{proof} We start by noting that the morphism $[n]$ extend to a morphism $[n]_{\can}:\can\rightarrow \can$, but not on the whole of $\caa$. Let us denote $\cau$ the domain of $[n]$, i.e. the largest open subset of $\caa$ such that $[n]$ extends to a morphism $[n]_{\cau}:\cau \rightarrow\caa$. Notice that, since $\caa$ is smooth and projective, the codimension of $Z:=\caa\setminus\cau$ is at least 2.

Let $S$ be the surface containing the pencil of curves and, in particular $C$ and $C'$. There is a rational function $f$ on $S$ such that, as $1$-cycles on $S$, we have $C'=C+\divv(f)$. Restricting to $S\cap\cau$ and applying $[n]$ we get $[n](C'\cap\cau)=[n](C\cap\cau)+[n](\divv(f))_{S\cap\cau}$. Now the Zariski closure of $[n](C'\cap\cau)$ (resp. $[n](C\cap\cau)$) is what we denoted $C'_0=[n]({C'})$ (resp. $C_0=[n]({C}))$. By proposition 1.4 in \cite{fu}, the image of 
$(\divv(f))_{S\cap\cau}$ by $[n]$ is a divisor $(\divv(g))_{[n](S\cap\cau)}$ for some function $g$ (in fact the norm of $f$, the norm being taken from the finite extension of corresponding function fields). We thus get, as $1$-cycles on $S'=[n](S)$ the equality
$C_0'=C_0+\overline{\divv(g))_{[n](S\cap\cau)}}$, from which we see that $C_0'=C_0+\divv(g)_{S'}+Z$ with $Z$ a $1$-cycle supported on $S'\setminus\cav$. But notice that the complement of $\cav$ has codimension at least 2, hence $Z=0$ and we have shown that there exists a surface $S'$ containing $C_0$ and $C_0'$ and a rational function $g$ on $S'$ such that:
$$C_0'=C_0+\divv(g).$$
Now $C_0$ and $C_0'$ are rationally, hence numerically equivalent and hence, by Lemma \ref{ntint}, the N\'eron-Tate heights of the associated points $P_0$, $P_0'$ are equal.
\end{proof}
 
\begin{remark} For elliptic surfaces over $\bbp^1$ linear equivalence tensored with $\bbq$ is the same as numerical equivalence, thus the previous lemma can be stated for numerically equivalent curves, as in \cite{sal}. We do not know if this generalisation holds in the higher dimensional case.
\end{remark}

\section{A key lemma}\label{key}

The following setting will be key to our results. Let $\pi:\caa \rightarrow B$ be a fibration of abelian varieties defined over a number field $k$. Consider a smooth projective curve $C$ and a map $i:C\rightarrow\caa$ such that $\psi:=\pi\circ i$ is a finite morphism, we denote $\pi_C:\caa_C\rightarrow C$ the new 
fibration of abelian varieties obtained by desingularizing the fibered product $\caa\times_{\psi,B}C$ and the obvious map. The fibered product is the scheme (algebraic variety)  fitting in the diagram  
\begin{eqnarray}
\xymatrix{ 
&  \caa_C\ar[d] &\\ &  \caa\times_{\psi,B}C \ar[dr]^{\pi_C} \ar[dl]_{\pi_1} 
& 
\\ \caa \ar[dr]_{\pi} & &C \ar[dl]^{\psi} \\
 & B  &}
\end{eqnarray}
and such that any pair of maps $f_1:X\rightarrow\caa$ and $f_2:X\rightarrow C$ such that $\pi\circ f_1=\psi\circ f_2$ induces a map $f:X\rightarrow\caa\times_BC$ such that $f_1=\pi_1\circ f$ and $f_2=\pi_C\circ f$. If further $X$ is a smooth curve, the map $ X\buildrel{f}\over{\rightarrow}\caa\times_BC\dashrightarrow\caa_C$ extends to a morphism $\tilde{f}:X\rightarrow\caa_C$.
Any section $\sigma: B\rightarrow\caa$ (i.e. a map such that $\pi\circ \sigma=id_B$), induces a section $\tilde{\sigma}:C\rightarrow \caa_C$. We thus get an injection
$A(k(B))=\caa(B)\rightarrow\caa_C(k({C}))$. We call {\it old sections} the sections of $\caa_C\rightarrow C$ obtained in this way. The maps $i:C\rightarrow\caa$ and $\psi:C\rightarrow B$ induce also a section $\sigma_C: C\rightarrow\caa_C$. If $\psi$ is of degree $\geq 2$, i.e. if $i({C})$ is {\it not} the image of a section, we call $\sigma_C$ a {\it new section}, or the new section induced by $C$ (and $\psi$). 

Now, the new section is not necessarily linearly independent of the group of old sections: it does not belong to the group of old sections but a positive multiple could belong to the group of old sections. The following key lemma provides a criterion which guarantees this does not happen too often, this lemma and its proof  are parallel to Proposition 4.2 in \cite{sal}. 

 \begin{lemma}\label{prpimportante}
 Let $\caa \rightarrow B$ be an abelian fibration defined over a number field $k$, whose fibres are abelian varieties of dimension $g$. Let $\cac$ be a pencil of curves inside $\caa$, no member being contained in a fiber. There exist an $n_0=n_0(\cac) \in \mathbb{N}$ and a finite subset $\Sigma_0=\Sigma_0(\cac)\subset Sec(\caa)$ such that for $C$ in the pencil $\cac$, the new section induced by $C$
 is linearly dependent of the old ones if, and only if, $[n]C \in \Sigma_0$ for some $n\leq n_0$.  
 \end{lemma}

\begin{proof}  Suppose $[n]C= C_0$ for some section $C_0$. We may assume such $n$ to be minimal i.e., there does not exist $n'<n$ such that the curve $[n']C$ is a section. 
The proof is divided in two parts:
\begin{itemize}
\item[1)] Bounding $n$ from above using Kummer theory (see Section \ref{kummer}).
\item[2)] For a fixed $n$ such that $C_0=[n]C$, showing that the  N\'eron-Tate height of such a section is bounded in terms of $n$ and the
pencil $\cac$ (see Section \ref{heightsection}).
\end{itemize}

\textbf{1) Bounding n:} We define the degree of a curve in $\caa$ by its intersection number with a fiber:
\[ \mathrm{deg}(C)= (C.F).\] 

 If $C_0$ is a section we have $\mathrm{deg}(C_0)=(C_0.F)=1$. The degree of $C$, which will be denoted by $h$, is fixed when $C$ belongs to a given  numerical class and {\it a fortiori} when $C$ moves in a pencil.
 
The map $[n]$ is not a morphism defined on the whole variety $\caa$, but on an open set $\cau \subseteq \caa$ that contains $\can$ the N\'eron model (say) but which excludes part of the boundary of $\can$ contained in the fibres with bad reduction. Since sections do not intersect this boundary, they are contained in $\cau$. This allows us to write:
  \[\mathrm{deg}([n]^{-1}C_0)= (([n]^{-1}C_0).F)= n^{2g}(C_0.F)= n^{2g}.\]

 Thus, $\lim_{n \to \infty}\deg [n]^{-1}(C_0) = \infty$. But notice that $C$ is an irreducible component of $[n]^{-1}(C_0)$; we now use Kummer theory (see Section \ref{kummer}) to show that $[n]^{-1}(C_0)$ cannot have components of small degree if $n$ is large.
 
 Denote by $K$ the field $k(B)$, by $A$ the generic fibre of $\caa$ and by $P_0$ the point in $A(K)$ corresponding to the section $C_0$.  Let $P\in A(\bar{K})$ be such that \[ [n] P=P_0 \] where $n$ is minimal with respect to the expression above. We now show that $n$ is bounded by a constant $n_0$ that depends only on $A$, $K$.
 
  Note first that if $P$ is a torsion point of order $m$, we know by Lemma \ref{serre} that  $ c_A.m^{\gamma}\leq [K(P):K]= h$, and $m$ is therefore bounded. 
    
  Now suppose $P$ is of infinite order. Let $m$ be the smallest positive integer such that there exist $P_1 \in A(K)$ and $T$ a torsion point satisfying \[ [m] P=P_1 + T.  \]
 We claim that $P_1$ is indivisible by $m$. If $l$ is a divisor of $m$ such that $[l]Q = P_1$ with $Q \in A(K)$, then \[[l][\frac{m}{l}]P=[l]Q +T\] and hence there exists a torsion point $T_1$ such that \[[\frac{m}{l}]P= Q+T_1.\] By the minimality of $m$ with respect to the equation above we must have $l=\pm 1$. 
 
  Let $m'=mm_1$ where $m_1$ is the order of the torsion point $T$. Let $T'$ be such that $[m]T' = T$ and put 
$P' = P + T'$ ; then $[m]P' = P_1$ and $T' \in A_{m'}$ . Applying Proposition \ref{kt} we obtain
  \begin{equation}\label{kum}  
   [K(P',A_{m'}): K(A_{m'})] \geq m.f_1.
  \end{equation}
  Now, note that $K(P, A_{m'} ) = K(P' , A_{m'} )$ hence \[ h=[K(P ) : K] 
\geq [K(P, A_{m'} ): K(A_{m'} )] = [K(P' , A_{m'} ): K(A_{m'} )] \geq f_1 .m,\] 
and thus m is bounded. Since $T$ is defined over $K(P )$, its order $m_1$ is also bounded in terms of $h$, and thus $[mm_1] P = 
m_1 P_1\in A(K)$ with $n\leq mm_1$ bounded as announced.

  \textbf{2) The equivalence class of a section:}
  Let now $C$ be a curve  member of a pencil $\cac$. For a fixed $n$, if $[n]({C})$  is a section, say $C_0$, then, by Lemma \ref{numclass},  its N\'eron-Tate height only depends on the pencil and thus the section belongs to a finite set $\Sigma_0$ as announced.
  \end{proof} 
  
  \medskip  

The following corollaries are immediate consequences of Lemma \ref{prpimportante} and our setting. 

\begin{corollary}\label{principal}
Let $\pi:\caa\rightarrow B$   be an abelian fibration and suppose $\caa$ contains a $k$-unirational surface $S$ not contained in a fibre of $\pi$. Then there exists a $k$-rational curve $C$ and a finite morphism $\phi:C\rightarrow B$ such that the base change raises the rank; in symbols, denoting $\caa_C=\caa\times_BC$, we have:
\begin{equation}
\rk\caa_C\left(k(C)\right)\geq \rk \caa\left(k(B)\right)+1.
\end{equation}
\end{corollary}
\begin{proof}
By hypothesis there exists a dominant map defined over $k$, say $\bbp^2 \dashrightarrow S$. This allows us to construct a pencil of rational curves on $S$ (e.g. take the image of a pencil of lines on $\bbp^2$) and, applying Lemma \ref{prpimportante}, we see that for almost all $C$ in this pencil, choosing the base change given by $C\hookrightarrow S\rightarrow B$, the new section associated to $C$ will be independent of the old sections and the rank will be raised by (at least) one.
\end{proof}

 \begin{corollary}\label{cor2} Let $\pi:\caa\rightarrow B$ be as in Corollary \ref{principal}, there exists infinitely many $t\in B(k)$ such that:
\begin{equation}
\rk\caa_t\left(k\right)\geq \rk \caa\left(k(B)\right)+1.
\end{equation}
\end{corollary}
\begin{proof} We apply Silverman's specialisation theorem (Theorem \ref{specialise}) to the family $\caa_C$ provided by the previous corollary and get that, for almost $u\in C(k)$ we have
$$\rk (\caa_C)_u(k)\geq \rk\caa_C(k(C))\geq \rk\caa(k(B))+1$$
But clearly $(\caa_C)_u=\caa_{\phi(u)}$ and thus we get for all $t=\phi(u)\in\phi(C(k))$ that:
$$\rk\caa_{t}(k)\geq \rk\caa(k(B))+1.$$
\end{proof}

\begin{remark}
These corollaries clearly imply Theorems \ref{theo1} and \ref{theo1bis}.

\end{remark}

\section{Families of Jacobians}\label{famille}
 
We use Milne's survey {\it Jacobian varieties} \cite{mijv} as reference.  Since we assume (for simplicity) that a curve $C$ defined over a field $K$ possesses a $K$-rational point $P_0$, we may identify its jacobian $\mathrm{Jac}({C)}$ with  $\pic^0({C)}$ and the jacobian embedding $\varphi: C\rightarrow \mathrm{Jac}({C)}$ with the map $P\mapsto \cl\left(P-P_0\right)$.

\begin{proposition}
Let $K$ be a field and $C$ a smooth projective curve of genus $g\geq 1$ defined over $K$ such that $C(K) \neq \emptyset$. Then the \textit{canonical} map $\varphi: C \rightarrow \mathrm{Jac}(C)$ is an injection. In particular, $\varphi(C)$ is a smooth projective curve. Moreover $\varphi(C)$ is not contained in any proper abelian subvariety of $\mathrm{Jac}(C)$.
\end{proposition}

We denote by $\sym^rC=C\times\dots\times C/\frS_r$ the symmetric product, i.e. the quotient of $C^r=C\times\dots\times C$ by $\frS_r$, the group of permutations on $r$ letters. This is a smooth projective variety of dimension $r$ (see, for example, \cite{mijv} Proposition 3.2). We can generalize the map $\varphi: C \rightarrow \mathrm{Jac}(C)$ as follows: the map from $C\times\dots\times C$ to $ \mathrm{Jac}(C)$ defined by $(x_1,\dots,x_r)\mapsto\varphi(x_1)+\dots+\varphi(x_r)$ is invariant under the action of  $\frS_r$ and thus induces a map $\varphi_r:\sym^rC\rightarrow\mathrm{Jac}(C)$ whose image is classically denoted by $W_r({C})$; further, when $r\leq g$, the map $\varphi_r:\sym^rC\rightarrow W_r({C})$ is birational (see \cite{mijv} Theorem 5.1). In particular $W_{g-1}=\Theta_C$ is the so-called theta divisor and the map  $\varphi_g:\sym^gC\rightarrow\mathrm{Jac}(C)$  is birational (this is the starting point of Weil's construction of the jacobian).

More precisely, if we identify points of $\sym^rC$ with effective divisors on $C$ with degree $r$, we have $\varphi(D_1)=\varphi(D_2)$ if and only the two divisors $D_1$ and $D_2$ are linearly equivalent.
All the constructions above can be generalized over a base $B$, which for convenience we will choose to be a curve. Thus we consider now a smooth projective surface $\mathcal{X}$ equipped with a fibration
$\pi : \mathcal{X}\rightarrow B$ whose fibres are curves of genus $g\geq 1$. We obtain thus a jacobian fibration $\tilde{\pi} : \mathcal{J}\rightarrow B$, first over the open subset $U$ above which the fibres $\mathcal{X}_b$ are smooth and then extended in some fashion over $B$.

The group $\frS_r$ acts on the fibre product $\mathcal{X}\times_B\dots\times_B\mathcal{X}$ and we denote the quotient:
$$\sym_B^r\cax:=\mathcal{X}\times_B\dots\times_B\mathcal{X}/\frS_r$$
We obtain a $B$-morphism birational onto its image
$$\varphi_r=\varphi_{r,B}: \sym_B^r\cax \rightarrow \mathcal{J}$$
\begin{lemma}\label{surf}
If $S\subset\sym_B^r\cax$ is a surface projecting onto $B$ and the generic fibre of $S$ over $B$ is a curve of genus $\geq 1$, then the image by $S':=\varphi_r(S)$ inside $\caj$ is again a surface. If further $S$ is $k$-unirational, then so is $S'$.
\end{lemma}
\begin{proof}
Since the generic fibre of the morphism $S\rightarrow B$ is a curve of genus at least one, the map $\varphi_r$ (restricted to this generic fibre) cannot be constant since this would say that the curve is parametrized by the projective line. Thus $S'=\varphi(S)$ surjects onto $B$ by hypothesis and that surjection has a generic fibre of dimension one, therefore $\dim(S')=2$.
\end{proof}

We now can prove Theorem \ref{theo2bis}.

\begin{proof} (of Theorem \ref{theo2bis})
The map $f:\cax\rightarrow \cax_0$ is by hypothesis a $B$-map with generic degree $d$ thus induces a
$B$-map $\cax_0\rightarrow \sym_B^d\cax$, whose image $S$ is $k$-unirational, since $\cax_0$ is itself assumed to be $k$-unirational. Applying Lemma \ref{surf} we see that the image of $S$ in $\caj$ is again  a $k$-unirational surface and applying (the proof of) Theorem \ref{theo1bis}, we conclude that there exists a base change by a $k$-rational curve $C$, say $\phi:C\rightarrow B$, which raises the rank.
\end{proof}

\section{examples}\label{examples}

We display in this section concrete examples of application of our results as well as examples falling slightly outside the  general setting. We will in particular make use of the two following geometric statements.

\begin{lemma}\label{lemw} (Weil) Let $\pi:\caa \rightarrow U$ be an abelian scheme over a smooth base $U$, defined over the field $k$. Any rational section of $\pi$ extends to a morphism $s: U\rightarrow \caa$. The group of morphic sections $\Gamma(U,\caa)$ is isomorphic to $A(k(U))$.
\end{lemma}
\begin{proof} The main point is that a rational map from a smooth variety to a projective variety  has a set of indeterminacy of codimension at least two, whereas, by a classical result of Weil,  a rational map from a smooth variety to a group variety has a set of indeterminacy of pure codimension one. The extension to group schemes is carried out for example in Proposition 1.3 of \cite{ar}.
\end{proof}

The next lemma gives the relation between the rank of the N\'eron-Severi group of a surface $S$, denoted $\ns(S)$ and the rank of the jacobian of the generic fibre of a fibration $\pi:S\rightarrow B$.

\begin{lemma}\label{shio3} (Shioda) Let $S$ be a surface defined over an algebraically closed field $\bar{k}$ and fibered over a curve $B$. Let $S_{\eta}$ be the curve over $k(B)$ which is the generic fibre of $\pi: S\rightarrow B$. Assume $\pi$ has a section $C_0\subset S$ and denote $\jac_{X_{\eta}}(k(B))$ the Mordell-Weil group of the Jacobian. We have an exact sequence
\begin{equation}
0\longrightarrow \mathcal{T}\longrightarrow \pic(S)\longrightarrow \jac_{X_{\eta}}(k(B))\longrightarrow 0,
\end{equation}
where $\mathcal{T}$ is the subgroup generated by the section $C_0$ and the components of fibres of $\pi$.
In particular when the $K/k$-trace is zero we have $J(K)\cong\ns(S)/\mathcal{T}$ and, denoting $\rho$ the rank of $\ns(S)$ and $r$ the rank of $J(K)$ and $m_v$ the number of components of the fibre $S_v$, we have the Shioda-Tate formula:
\begin{equation}\label{st}
\rho=r+2+\sum_{v\in B}(m_v-1).
\end{equation}
\end{lemma}
\begin{proof} (sketch and references) See \cite{shio2} or \cite{hp} or, for a more general setting, see \cite{hpw}. The main point is to define the map $\pic(S)\longrightarrow \jac_{X_{\eta}}(k(B))$ and determine its kernel. This map   associates to the class of a divisor $D\in\pic(S)$ the restriction to the generic fibre of the class of $D-\deg(D)C_0$, where $\deg(D)$ is the intersection number of $D$ with a fiber.
\end{proof}

\subsection{Unirational families of abelian varieties}

Several authors have recently showed that certain families of abelian varieties form geometrically unirational varieties, i.e., that are unirational over an algebraically closed field (see \cite{verra}, \cite{gp}). Since these varieties are defined over $\mathbb{Q}$, there exists a number field  $k$ such that they are unirational over $k$. Our results or more precisely  slight variations of our results  apply to these. 

For given $g$ we denote $\caa_g$ the coarse moduli space of prinicipally polarized abelian varieties of dimension $g$ and $\pi:\cax_g\rightarrow\caa_g$ the universal family, which is well defined over an open subset of $\caa_g$.
 We obtain, in particular, by a variation on Theorem \ref{theo1} applied to the unirational universal family of abelian varieties $\cax_4$ and $\cax_5$, of dimensions 4 and 5, respectively, the following result. 

\begin{proposition}\label{example1} For $g=4$ or $5$, there are a number field $k$ and a Zariski dense set $T\subset \caa_g(k)$ such that for all $t\in T$, we have $\rk \cax_t(k)\geq 1$.
\end{proposition}
The density can be made a bit more precise, namely $T$ can be taken as the image by a dominant morphism $Y\rightarrow\caa_g$ of the complement of a thin set of $Y(k)$.

\begin{proof} Let $\pi:\cax_g\rightarrow\caa_g$ denote the universal family of abelian varieties for $g=4$ or $5$, we have a dominant map $\psi:\bbp^N\rightarrow \cax_g$ where $N:=\dim \cax_g=g(g+3)/2$ and we can pick a pencil of linear subspaces $L_t\subset\bbp^N$ with dimension 
$M:=\dim \caa_g=g(g+1)/2$, such that the composite map $L_t\cdots\rightarrow \caa_g$ is dominant. In particular $\psi(L_t)$ can be viewed as a multisection of $\pi:\cax_g\rightarrow\caa_g$. Using arguments similar to those used to prove Theorem \ref{theo1bis}, one concludes that for most $t$ the multisection cannot be torsion. Applying N\'eron's theorem (Theorem \ref{specialise}, first item) instead of Silverman's specialisation theorem yields the proposition.
\end{proof}

\subsection{Shioda's family of hyperelliptic curves}\label{shioex}
Many examples of elliptic surfaces are displayed in \cite{sal2}. In order to consider examples of higher genus, we start with an example studied in \cite{shio2,shio3}.
 Let $X$ be the hyperelliptic curve of genus $g$ over $K=k(t)=k(\bbp^1)$ studied in \cite[Section 5.1]{shio2} given by the equation

\begin{equation}
y^2=x^{2g+1}+p_2x^{2g}+\dots+p_{2g}x+p_{2g+1}+t^2=p(x)+t^2
\end{equation} 
where $p_i$ belong to $k$.

We will assume that $p(x)$ is separable and split over $k$, i.e.  there exist distinct $e_i$'s in $k$ such hat $p(x)=(x-e_1)\dots(x-e_{2g+1})$, we will reprove that the  Mordell--Weil group 
 $\caj(K)$ of the Jacobian variety $\caj$ has rank $2g.$ 
Notice that the algebraic surface $\cax$ associated to $X$ is a rational surface with a unique reducible fibre at infinity.

We consider $K_0=k(u)$ and put $t^2=u$ and $K=K_0(t)=k(t)$ so that $K/K_0$ is a quadratic extension with Galois group generated by $\sigma$ (i.e. $\sigma(t)=-t$). We introduce now the curves
$$X_0/K_0: \quad y^2=p(x)+u\qquad{\rm and}\qquad X/K: \quad y^2=p(x)+t^2$$
and their respective Jacobians $J_0=\jac(X_0/K_0)$ and $J=\jac(X/K)$
so that $X=X_0\times_{K_0}K$ and $J=J_0\times_{K_0}K$. Both curves have a ``point at infinity" which we denote by $\infty$ and we choose the Jacobian embedding
$j:P\mapsto cl(({P})-(\infty))$.
The curve $X$ has the obvious points $P_i=(e_i,t)$ and $P'_i=(e_i,-t)$ for $i=1,\dots,d$.
One sees easily that
$$\divv(x-e_i)=(P_i)+(P'_i)-2(\infty)\;{\rm and}\;\divv(y-t)=(P_1)+\dots+(P_d)-d(\infty)$$
hence we have the relations 
\begin{equation}\label{relp} j(P'_i)=-j(P_i)\quad{\rm and}\quad j(P_1)+\dots+j(P_d)=0.
\end{equation}
\begin{proposition}\label{example2} The following statements hold:
\begin{enumerate}
\item The $K/k$-trace of $J$ is trivial. 
\item The group $J_0(K_0)$ has rank zero.
\item The group $J(K)$ has rank $2g$ and the points $j(P_1),\dots,j(P_{2g})$ generate a subgroup of finite index; in fact
$2g$ is equal to the rank of $J(\bar{k}(t))$. 
\item For infinitely many $t\in k$, the group $J_t(k)$ has rank at least $2g+1$.
\end{enumerate}
\end{proposition}
\begin{proof} The surface $y^2=p(x)+t^2$ is covered by rational curves (conics) hence has trivial Albanese and Picard varieties. The $K/k$-trace of $J$ is a quotient of the Picard variety (Cf. \cite{hp,shio2,shio3}) and hence is trivial.

Let $\Gamma$ be the group generated by the $j(P_i)$'s and $j(P_i')$'s. We will check that (\ref{relp}) are the only relations between the points $j(P_i)$, $j(P'_i)$. 

The specialisation at $t_0=0$ sends $\Gamma$ onto the kernel of multiplication by two on $J_{t_0}$, the Jacobian of the curve $y^2=p(x)$, which is isomorphic to $(\bbz/2\bbz)^{2g}$. Since the specialisation is injective on torsion subgroups and $\Gamma$ is a quotient of $\bbz^{2g}$, we deduce that $\Gamma\cong \bbz^{2g-s}\times (\bbz/2\bbz)^{s}$. But one checks easily that non-trivial 2-torsion points on $J$ are not defined over $k(t)$, indeed the polynomial $p(x)+t^2$ is clearly irreducible in $k(t)[x]$.
 
To compute the geometric rank $r$, we introduce the minimal model $f:S\rightarrow\bbp^1$ of $X/K$. The singular fibres are at $t=\infty$ with, say, $m_{\infty}$ components and $2(d-1)$ singular, irreducible fibres corresponding to  $t=\pm\sqrt{-p(\alpha)}$ for $p'(\alpha)=0$, thus the so-called Shioda-Tate formula (Lemma \ref{st}) for $\rho$  the rank of $\ns(S)$ reads
\begin{equation}
\rho=r+m_{\infty}+1.
\end{equation}
Writing the (affine) equation as $y^2-t^2=p(x)$ we can view $S$ as a conic bundle with $d$ degenerate fibres meeting this affine part. Applying again Lemma \ref{st},  this is enough to compute
$\rho-m_{\infty}=d$ and hence $r=2g$.

To compute the rank of $J_0(K_0)$, recall $\sigma$ is the generator of the Galois group of $K/K_0$, i.e. the automorphism of $K$ fixing $K_0$ and sending $t$ to $-t$. The group $J_0(K_0)$ can be identified with the subgroup of $J(K)$ fixed by $\sigma$, but $\sigma(j(P_i))=j(P'_i)=-j(P_i)$ and therefore $\sigma$ acts as multiplication by $-1$ on $\Gamma$. This implies that $\sigma$ has no fixed points other than $0$ in $\Gamma$ and therefore the rank of  $J_0(K_0)$ is zero.

Thus we get an example where a base change $\bbp^1\rightarrow\bbp^1$ produces an augmentation $+2g$ of the rank. We can exploit the geometry to obtain an increase of $2g+1$ as follows. The surface $S$ has the structure of a conic bundle hence a pencil of rational curves $C_a$ with an affine equation $y^2-t^2=p(a)$, applying the base change technique to this pencil amounts to using the extension $K_1:=k(t,v)$ where $v^2=p(a)+t^2$, being careful to choose $a\not=e_i$; in other words $K_1$ is a biquadratic extension of $K_0$ which we may write $K_1=K_0(\sqrt{u},\sqrt{p(a)+u})$. We automatically get a new section $P_a=(a,v)$. Suppose that $j(P_a)$ is linearly dependent on the old sections, i.e. we have a relation  $mj(P_a)=\sum m_ij(P_i)$ with $m\geq 1$.
Consider the element $\sigma\in\gal(K_1/K_0)$ which fixes $v=\sqrt{p(a)+u}$ and verifies $\sigma(t)=-t$. We have $\sigma(j(P_a))=j(P_a)$ and $\sigma(j(P_i))=-j(P_i)$ and therefore $mj(P_a)=\sum m_ij(P_i)$ implies $2mj(P_a)=0$. There is only a finite number of values of $a$ for which $P_a$ can be torsion so, for almost all $a$ the rank of $J_1=J_0\times_{K_0}K_1$ over $K_1$ is $\geq 2g+1$.
Notice that $K_1$ is the function field of the curve defined by the equation   $v^2=p(a)+t^2$ which is a conic with a rational point, hence isomorphic to $\bbp^1$ over $k$. \end{proof}

\begin{remark} The case when $p(x)$ has even degree is slightly different. If say
$p(x)=(x-e_i)\dots (x-e_{2g+2})$ is separable over $k$ and $X/k(t)$ is given by the equation $y^2=p(x)+t^2$, we have again the points $P_i=(e_i,t)$ and $P_i'=(e_i,-t)$, but we have now two points at infinity $\infty_1$ and $\infty_2$ and 
$\divv(x-e_i)=(P_i)+(P'_i)-(\infty_1)-(\infty_2)$ and $\divv(y-t)=(P_1)+\dots+(P_{2g+2})-(g+1)((\infty_1)+(\infty_2))$
If we pick $\infty_2$ as origin, i.e. define $j(P)$ as th class of $(P)-(\infty_2)$,  we get the relations 
\begin{equation}\label{relpbis} j(P'_i)+j(P_i)=j(\infty_1)\quad{\rm and}\quad j(P_1)+\dots+j(P_{2g+2})=(g+1)j(\infty_1). 
\end{equation}
In general, i.e. when $p(x)$ is sufficiently generic these are the only relations and the $j(P_i)$'s generate a group isomorphic to $\bbz^{2g+2}$ and of finite index in
$J_X(k(t))$ or even in $J_X(\bar{k}(t))$. But there are special cases, e.g. when $p(x)=x^{2g+2}+a$, since then $\divv(y-x^n)=(g+1)(\infty_1)-(g+1)(\infty_2)$, so $j(\infty_1)$ is torsion in this case.
\end{remark}

\subsection{A family with non-trivial trace} Our next example is a variant and covering of Shioda's family of hyperelliptic curves.
Let $p(x)$, $q(x)$ be  polynomials in $k[x]$ of degrees $d_1=2g_1+1$ and $d_2=2g_2+2$ respectively; we will further assume that $p(x)q(x)$ is separable and that $p(x)=\prod_i(x-e_i)$ with $q(e_i)=a_i^2$, where $e_i,a_i$ belong to $k$; we write $q(x)=q_0(x-f_1)\dots(x-f_{d_2})$ but do not assume that the $f_j$'s belong to $k$. We let again $K=k(t)$ and consider the curve $X/K$ with   affine equations in $\bba^3$ given by
\begin{equation}\label{bicurve}\left\{\begin{matrix} y^2=p(x)+t^2\cr z^2=q(x)\cr\end{matrix}\right.
\end{equation}
The curve $X$ has two points at infinity and its genus is $d_1+d_2-2$.
If $X_1$ is the curve considered in the previous example with affine equation $y^2=p(x)+t^2$ and $X_2$ the curve with affine equation $z^2=q(x)$ and finally
$X_3$ the curve with affine equation $w^2=(p(x)+t^2)q(x)$, one checks easily that $J_X$, the Jacobian of $X$, is isogenous to the product of the Jacobians of $X_1$, $X_2$ and $X_3$. Clearly $J_{X_2}$ descends to $k$; we have seen that the $K/k$-trace of $J_{X_1}$ is trivial; since the surface $w^2=(p(x)+t^2)q(x)$ is covered by conics, the $K/k$-trace of $J_{X_3}$ is also trivial and finally the $K/k$-trace of $J_X$ is isogenous to $J_{X_2}$.
By the previous example, the rank of $J_{X_1}(K)$ is $d_1-1=2g_1$. 

We introduce now the following points on $X_3$ as $Q_j=(f_j,0)$ for $q(f_j)=0$ and $R_i=(e_i,a_it)$, $R_i'=(e_i,-a_it)$. Let us check that they generate a group isomorphic to $\bbz^{d_1}\times(\bbz/2\bbz)^{d_2}$. Let $\Gamma$ be the subgroup of $J_{X_3}(\bar{k}(t))$ generated by
the $j(Q_j)$'s and $j(R_i)$'s. We first notice that, on $X_3$, we have $\divv(x-f_j)=2(Q_j)-2(\infty)$ and $\divv(x-e_i)=(R_i)+(R'_i)-2(\infty)$ so $2j(Q_j)=0$ and $j(R_i)+j(R'_i)=0$. Thus the $j(Q_j)$ generate a subgroup $(\bbz/2\bbz)^{d_2}$.
%.\check{Terminer (et v\'erifier les calculs pour cet exemple; use Lemma \ref{st}}
Next we show that the points $j(R_i)$ are independent: if $\Gamma_0$ is the subgroup generated by the $j(R_i)$'s, we have $\Gamma_0\cong\bbz^{d_1}$. For this we observe again that, via specialisation at $t_0=0$, the points $R_i$ specialise to Weierstrass points and therefore $j(R_i)$ specialises to a point of order two; in fact, if we denote $j_0(P)$ the specialisation of $j(P)$, we see that the $j_0(Q_j)$ and $j_0(R_i)$ generate $J_{t_0}[2]\cong(\bbz/2\bbz)^{d_2+d_1-1}$, with the unique relation among them being $\sum_jj_0(Q_j)+\sum_ij_0(R_i)=0$. We see therefore that  $\Gamma_0$ maps onto a group isomorphic to $(\bbz/2\bbz)^{d_1}$.
Since specialisation is injective on torsion, we see that $\Gamma_0\cong \bbz^{d_1-s}\times(\bbz/2\bbz)^{s}$. But, since $p(x)+t^2$ is irreducible in $K[x]$, we see that
$J[2](K)$ is the subgroup generated by the $j(Q_j)$ and thus $s=0$.

We conclude that:
$$\Gamma:=<\dots R_i, \dots, Q_g,\dots>\cong  (\bbz/2\bbz)^{d_2}\times\bbz^{d_1}$$
To compute the rank of $J_{X_3}(K)$ we rewrite the curve $X_3$ as $q(x)v^2=p(x)+t^2$, which we obtain by putting $(v,x,t)=(w/q(x),x,t)$. We apply again Lemma \ref{st} as follows: the surface can be viewed as a conic bundle over the $x$-line that has $d_1+d_2+1$ singular fibres, corresponding to the zeroes 
of $q(x)p(x)$ and the point at infinity. Therefore it's Picard rank is $\rho=2+(d_1+d_2+1)$. Comparing with Shioda-Tate formula applied to the surface fibered over the $t$-line we obtain that
$\rk J_{X_3}(\bar{k}(t))=d_1$. Summing up,  we have shown that, denoting $\tau:A_0\hookrightarrow J_X$ the trace:

\begin{proposition}\label{example3} Let $X/K$ be the curve given by equations (\ref{bicurve}).
\begin{enumerate}
\item The $K/k$-trace of $J_X$, denoted $A_0$, is isogenous to $J_{X_2}$, in particular it has dimension $g_2=\frac{d_2}{2}-1$.
\item The generic rank is given by:
$$
\rk J_X(K)/\tau A_0(k)=\rk J_{X_1}(K)+\rk J_{X_3}(K)=2d_1-1=4g_1+1.
$$
We also have $\rk J_X(K)=\rk J_{X_2}(k)+2d_1-1$.  
\item For almost all $a\in k$, the base change to $K_1=K(\sqrt{p(a)+t^2})$ raises the generic rank  at least by one.
\end{enumerate}
\end{proposition}

\medskip

\subsection{An example of double base change raising the rank by two.}
We examine the previous curve $X$ over $K=k(t)$. We have seen that for almost all $a\in k$,  over $K_1=K(\sqrt{p(a)+t^2})$, the rank of the Mordell-Weil group of $J_{X_1}$ is at least $\rk J_{X_1}(K)+1$. Similarly for almost all $b\in k$,  over $K_2=K(\sqrt{q(b)(p(b)+t^2)})$, the rank of the Mordell-Weil group of $J_{X_3}$ is at least $\rk J_{X_3}(K)+1$.  Hence we see that over $K_3=K_1K_2=K(\sqrt{p(a)+t^2},\sqrt{q(b)(p(b)+t^2)})$, the rank of the Mordell-Weil group of $J_{X}$ is at least $\rk J_{X}(K)+2$. However it is unclear if we can ensure that the curve $C$ such that $K_3=k(C)$ satisfies that $C(k)$ is infinite. Indeed we will prove this only under the restrictive conditions stated in the next lemma.

\begin{lemma}\label{tricky} Let $C_{a,b}$ be the curve defined by equations $r^2=p(a)+t^2$ and $s^2=q(b)(p(b)+t^2)$. Then we have:
\begin{enumerate}
\item For almost all $a,b\in k$, the curve $C_{a,b}$ has genus one.
\item Suppose further that:
\begin{enumerate}
\item The polynomial $p(t)q(t)$ is separable,  $\deg(p)=3$ and $\deg(q)=4$.
\item The polynomial $p$ has at least one $k$-rational root, say $e_1$ and $q(e_1)=a_1^2$ with $a_1\in k$.
\end{enumerate}
  Then for infinitely many values of $(a,b)\in k^2$, the set $C_{a,b}(k)$ is infinite.\
\end{enumerate}
\end{lemma}
\begin{proof}
The first part is the standard fact that a smooth intersection of two quadrics in $\mathbb{P}^3$
is  a curve of genus 1. 
For the second part, we consider the threefold $X$ defined in $\bba^5$ by the equations $r^2=p(a)+t^2$ and $s^2=q(b)(p(b)+t^2)$. The map $(r,s,t, a,b)\mapsto (s,t,b)$ sends $X$ onto the surface $Y_1$ defined in $\bba^3$ by $s^2=q(b)(p(b)+t^2)$. The fibres of the map $\phi$ from $X$ to $Y_1$  are elliptic curves; indeed if $(s,t,b)\in Y_1$ then $\phi^{-1}\{(s,t,b)\}:=X_{(s,t,b)}$ is the elliptic curve $E_t$ given by $r^2=p(a)+t^2$. The elliptic curve $E_t$, viewed over $k(t)$ has positive rank (the point $P_1=(e_i,t)$ is of infinite order, as we have seen in subsection \ref{shioex}) hence for almost all $t\in k$ the rank of $E_t(k)$ is also positive. The surface $Y_1$ is clearly birational to the surface $Y$ defined in $\bba^3$ by $q(b)s^2=p(b)+t^2$; the latter is an elliptic conic bundle of the type studied by Kollar and Mella \cite{kollar}, whose results imply that $Y$ is $k$-unirational. Indeed the surface is a conic bundle with 8 singular fibres (corresponding to $b$ equal to a root of $pq$ or to the point at infinity); but the fibre at $b=e_1$ has equation $a_1^2s^2-t^2=0$ and is clearly reducible; as explained in \cite{kollar}, the contraction of one of the two $k$-rational components of the fibre yields a conic bundle with 7 singular fibres, which is therefore $k$-unirational by the main result of  \cite{kollar}. The same is true for $Y_1$ and $X$ has therefore a Zariski dense subset of $k$-rational points. Consider now $f$ the fibration of $X$ defined by $(r,s,t, a,b)\mapsto (a,b)$; the fibres of $f$ are just the curves of genus 1 denoted $C_{a,b}$ in the lemma and it follows that, for infinitely many $(a,b)$ (actually a Zariski-dense subset)  the set $C_{a,b}(k)$ is infinite. 
\end{proof}

Combining this lemma with the previous computations we get an example where our base change technique provides infinitely many fibres of rank at least the generic rank plus 2.

\begin{proposition}\label{example4} Let $X$ be  the curve  over $k(t)$ defined by the pair of equations
$$y^2=p(x)+t^2,\quad {\rm and}\quad z^2=q(x).$$
Suppose $\deg p=3$ and $\deg q=4$ and $p(x)q(x)$ separable.  Assume $p(x)$ has a $k$-rational root $e_1$ with  $q(e_1)=a_1^2$ with $a_1\in k$, then
\begin{enumerate}
\item The $k(t)/k$-trace of $J_X$ is an elliptic curve $E_0$  isogenous to the curve with equation $z^2=q(x)$.
\item There exist infinitely many values $t\in k$ such that
$$\rk J_{X_t}(k)\geq \rk J_X(k(B))+2.$$
\end{enumerate}
\end{proposition}
Notice that if
$p(x)=(x-e_1)(x-e_2)(x-e_3)$ and $q(e_i)=a_i^2$ with $a_i,e_i\in k$, then the rank of $J_X(k(t))/\tau(E_0)(k)=2+\rk J_{X_3}(k(t))=5$.

\end{document}